\documentclass[review]{elsarticle}
\usepackage[english]{babel}
\usepackage{amsmath}
\usepackage{amssymb}
\usepackage{fancyhdr} 
\usepackage{bussproofs}
\usepackage{mathrsfs}

\usepackage{amsthm}
\usepackage{amsmath}%
\usepackage{amsfonts}%
\usepackage{amssymb}%

\usepackage{enumerate}

\newtheorem{theorem}{Theorem}[section]

\theoremstyle{definition}
\newtheorem{definition}[theorem]{Definition}

\newtheorem{proposition}{Proposition}[section]

\theoremstyle{definition} 

\newtheorem*{maintheorem*}{Main Theorem}

\begin{document}

\begin{frontmatter}

\title{$L$-Topology via Generalised Geometric Logic}




\author[mymainaddress]{Purbita Jana\corref{mycorrespondingauthor}}
\cortext[mycorrespondingauthor]{Corresponding author}
\ead{purbita\_presi@yahoo.co.in}

\address[mymainaddress]{The Institute of Mathematical Sciences (IMSc.), Chennai, India}

\begin{abstract}
This paper introduces a notion of generalised geometric logic. Connections of generalised geometric logic with L-topological system and L-topological space are established.
\end{abstract}


\end{frontmatter}


\section{Introduction}
This work is motivated by S. Vickers's work on topology via logic \cite{SV}. To show the connection of topology with geometric logic, the notion of topological system played a crucial role. A topological system is a triple $(X,\models, A)$, consisting of a non-empty set $X$, a frame $A$ and a binary relation $\models$ (known as satisfaction relation) between $X$ and $A$ satisfying certain conditions. The notion of topological system was introduced by S. Vickers in 1989. Topological system is an interesting mathematical structure, which unifies the concepts of topology, algebra, logic in a single framework. In our earlier work \cite{MP3}, we had introduced a notion of fuzzy geometric logic to answer the question viz. ``From which logic can fuzzy topology be studied?". For this purpose first of all we introduced the notion of fuzzy topological system \cite{PJ} which is a triple $(X,\models,A)$ consisting of a non-empty set $X$, a frame $A$ and a fuzzy relation $\models$ (i.e. $[0,1]$ valued relation) from $X$ to $A$. J. Denniston et al. introduced the notion of lattice valued topological system ($L$-topological system) by considering frame valued relation between $X$ and $A$. In \cite{DM1}, categorical relationship of Lattice valued topological space ($L$-topological space) with frame was established using the categorical relationships of them with $L$-topological system. Moreover categorical equivalence between spatial $L$-topological system with $L$-topological space was shown. In this paper the main focus is to answer the question viz. ``From which logic can $L$-topology be studied?". From \cite{MP3}, it is clear that the satisfaction relation $\models$ of fuzzy topological system reflects the notion of satisfiability ($sat$) of a geometric formula by a sequence over the domain of interpretation of the corresponding logic. Hence we considered the grade of satisfiabilty from $[0,1]$. As for $L$-topological system the satisfaction relation is an $L$ (frame)-valued relation, the natural tendency is to consider the grade of satisfiability from $L$. Keeping this in mind, generalised geometric logic (c.f. Section \ref{GGL}) is proposed to provide the answer of the raised question successfully.

The paper is organised as follows. Section \ref{pre}, includes some of the preliminary definitions and results which are used in the sequel. Generalised geometric logic is proposed and studied in details in Section \ref{GGL}. Section \ref{ltsy}, explains the connection of the proposed logic with $L$-topological system whereas Section \ref{lts}, contains the study of the connection of the proposed logic with $L$-topological space. Section \ref{con}, concludes the work presented in this article and provides some of the future directions.
\section{Preliminaries}\label{pre}
In this section we include a brief outline of relevant notions to develop our proposed mathematical structures and results. In \cite{MP3, CC, DM1, UH, MP, PT, SO,  SV, LZ} one may found the details of the notions stated here.
\begin{definition}[Frame]\label{fm}
A \textbf{frame} \index{frame} is a complete lattice such that, $$x\wedge\bigvee Y=\bigvee\{x\wedge y \mid y\in Y\}.$$
i.e., the binary meet distributes over arbitrary join.
\end{definition}
\begin{definition}[Fuzzy topological space]\label{3.1_1t}
 Let $X$ be a set, and $\tau$ be a collection of fuzzy subsets of $X$ s.t.
\begin{enumerate}
\item $\tilde\emptyset$ , $\tilde X\in \tau$, where $\tilde\emptyset (x)=0$, for all $x\in X$ and $\tilde X (x)=1$, for all $x\in X$;
\item $\tilde A_i\in\tau$ for $i\in I\ \text{implies}\ \bigcup_{i\in I}\tilde A_i\in \tau$, where $\bigcup_{i\in I}\tilde{A_i}(x)=sup_{i\in I}(\tilde{A_i}(x))$;
\item $\tilde A_1$ , $\tilde A_2\in\tau\ \text{implies}\ \tilde A_1\cap\tilde A_2\in\tau$, where $(\tilde{A_1}\cap \tilde{A_2})(x)=min\{\tilde{A_1}(x), \tilde{A_2}(x)\}$.
\end{enumerate}
Then $(X,\tau)$ is called a \textbf{fuzzy topological space}\index{fuzzy topological!space}. $\tau$ is called a \textbf{fuzzy topology} over $X$. 
\end{definition}
Elements of $\tau$ are called \textbf{fuzzy open sets}\index{fuzzy!open sets} of fuzzy topological space $(X,\tau)$.
\begin{definition}[$L$-topological space]\label{3.1_1t}
 Let $X$ be a set, and $\tau$ be a collection of $L$-fuzzy subsets of $X$ i.e., $\tilde{A}:X\to L$, where $L$ is a frame, s.t.
\begin{enumerate}
\item $\tilde\emptyset$ , $\tilde X\in \tau$, where $\tilde\emptyset (x)=0_L$, for all $x\in X$ and $\tilde X (x)=1_L$, for all $x\in X$;
\item $\tilde A_i\in\tau$ for $i\in I\ \text{implies}\ \bigcup_{i\in I}\tilde A_i\in \tau$, where $\bigcup_{i\in I}\tilde{A_i}(x)=sup_{i\in I}(\tilde{A_i}(x))$;
\item $\tilde A_1$ , $\tilde A_2\in\tau\ \text{implies}\ \tilde A_1\cap\tilde A_2\in\tau$, where $(\tilde{A_1}\cap \tilde{A_2})(x)=\tilde{A_1}(x)\wedge \tilde{A_2}(x)$.
\end{enumerate}
Then $(X,\tau)$ is called an \textbf{$L$-topological space}\index{L topological!space}. $\tau$ is called an \textbf{$L$-topology} over $X$. 
\end{definition}
Elements of $\tau$ are called \textbf{$L$-open sets}\index{L!open sets} of $L$-topological space $(X,\tau)$.
\begin{definition} \cite{SV}
 A \textbf{topological system}\index{topological!system} is a triple, $(X,\models,A)$, consisting of a non empty set $X$, a frame $A$ and a binary relation $\models\subseteq  X\times A$ from $X$ to $A$ such that:
\begin{enumerate}
\item for \textbf{any finite subset} $S$ of $A$, $x\models\bigwedge S$ if and only if $x\models a$ for all $a\in S$;
\item for \textbf{any subset} $S$ of $A$, $x\models \bigvee S$ if and only if $x\models a$ for some $a\in S$.
\end{enumerate}
\end{definition}
\begin{definition}[$L$-topological system]\label{3.1}
 An \textbf{$L$-topological system}\index{L topological!system} is a triple $(X,\models ,A)$, where $X$ is a non-empty set, $A$ is a frame and $\models$ is an $L$-valued relation from $X$ to $A$ ($\models:X\times A\to L$) such that
 \begin{enumerate}
\item if $S$ is a $\mathsf{finite}$ $\mathsf{ subset}$ of $A$, then
$gr(x\models \bigwedge S) = inf\{ gr(x\models s)\mid s\in S\}$;
\item if $S$ is  $\mathsf{any}$ $\mathsf{ subset}$ of $A$, then
$gr(x\models \bigvee S)=sup\{ gr(x\models s)\mid s\in S\}$.
\end{enumerate}
\end{definition}
For our convenience $\models(x,a)$ will be expressed as $gr(x\models a)$ throughout this article. It is to be noted that $\bigwedge S$ is either $a_1\wedge a_2\wedge \dots \wedge a_n$ if $S=\{a_1,a_2,\dots ,a_n\}$ and is $\top$ if $S=\emptyset$. Note that if $L=[0,1]$ then the triple is known as fuzzy topological system. 
\begin{definition}[Spatial]\label{spatial}
An $L$-topological system $(X,\models,A)$ is said to be \textbf{spatial}\index{spatial} if and only if (for any $x\in X$, $gr(x\models a)=gr(x\models b)$) imply ($a=b$), for any $a,b\in A$.
\end{definition}
\begin{theorem}\label{spatialtop}
Category of spatial $L$-topological systems, for a fixed $L$, is equivalent to the category of $L$-topological spaces.
\end{theorem}
\section{Generalised Geometric Logic}\label{GGL}
In this section we will introduce the notion of generalised geometric logic which may be considered as a generalisation of fuzzy geometric logic and consequently of so called geometric logic. Detailed studies on fuzzy logic, geometric logic and fuzzy geometric logic may be found in \cite{MP3, PH, SI, VN1, JP, SV1, SVI, PJT, SV}.

The \textbf{alphabet} of the language $\mathscr{L}$ of generalised geometric logic comprises of the connectives $\wedge$, $\bigvee$, the existential quantifier $\exists$, parentheses $)$ and $($ as well as:
\begin{itemize}
\item countably many individual constants $c_1,c_2,\dots$;
\item denumerably many individual variables $x_1,x_2,\dots$;
\item propositional constants $\top$, $\bot$;
\item for each $i>0$, countably many $i$-place predicate symbols $p^i_j$'s, including at least the $2$-place symbol ``$=$" for identity;
\item for each $i>0$, countably many $i$-place function symbols $f^i_j$'s.
\end{itemize}
\begin{definition}[Term]\label{term}
\textbf{Terms} are recursively defined in the usual way.
\begin{itemize}
\item every constant symbol $c_i$ is a term;
\item every variable $x_i$ is a term;
\item if $f_j$ is an $i$-place function symbol, and $t_1,t_2,\dots,t_i$ are terms then\\ $f^i_jt_1t_2\dots t_i$ is a term;
\item nothing else is a term.
\end{itemize}
\end{definition}
\begin{definition}[Geometric formula]\label{wff}
\textbf{Geometric formulae} are recursively defined as follows:
\begin{itemize}
\item $\top$, $\bot$ are geometric formulae;
\item if $p_j$ is an $i$-place predicate symbol, and $t_1,t_2,\dots,t_i$ are terms then $p^i_jt_1t_2\dots t_i$ is a geometric formula;
\item if $t_i$, $t_j$ are terms then $(t_i=t_j)$ is a geometric formula;
\item if $\phi$ and $\psi$ are geometric formulae then $(\phi\wedge\psi)$ is a geometric formula;
\item if $\phi_i$'s ($i\in I$) are geometric formulae then $\bigvee\{\phi_i\}_{i\in I}$ is a geometric formula, when $I=\{1,2\}$ then the above formula is written as $\phi_1\vee \phi_2$;
\item if $\phi$ is a geometric formula and $x_i$ is a variable then $\exists x_i\phi$ is a geometric formula;
\item nothing else is a geometric formula.
\end{itemize}
\end{definition}
\begin{definition}\label{t}
$t[t'/ x]$ is the result of replacing $t'$ for every occurrence of $x$ in $t$, defined recursively as follows:
\begin{itemize}
\item if $t$ is $c_i$ or $x_i$ other than $x$ then $t[t'/x]$ is $t$;
\item if $t$ is $x$ then $t[t'/x]$ is $t'$;
\item if $t$ is $f^i_jt_1t_2\dots t_i$ then $t[t'/x]$ is $f^i_jt_1[t'/x]t_2[t'/x]\dots t_i[t'/x]$.
\end{itemize}
\end{definition}
\begin{definition}\label{phi}
$\phi[t/x]$ is the result of replacing $t$ for every free occurrence of $x$ in $\phi$, defined recursively as follows:
\begin{itemize}
\item if $\phi$ is $p^i_jt_1t_2\dots t_i$ then $\phi[t/x]$ is $p^i_jt_1[t/x]t_2[t/x]\dots t_i[t/x]$;
\item if $\phi$ is $(t_i=t_j)$ then $\phi[t/x]$ is $(t_i[t/x]=t_j[t/x])$;
\item if $\phi$ is $\phi_1\wedge\phi_2$ then $\phi[t/x]$ is $\phi_1[t/x]\wedge\phi_2[t/x]$;
\item if $\phi$ is $\phi_1\vee\phi_2$ then $\phi[t/x]$ is $\phi_1[t/x]\vee\phi_2[t/x]$;
\item if $\phi$ is $\bigvee\{\phi_i\}_{i\in I}$ then $\phi[t/x]$ is $\bigvee\{\phi_i[t/x]\}_{i\in I}$;
\item if $\phi$ is $\top$ or $\bot$ then $\phi[t/x]$ is $\top$ or $\bot$ respectively;
\item if $\phi$ is $\exists x_i\psi$ ($x_i$ is other than $x$) then $\phi[t/x]$ is $\exists x_i\psi[t/x]$;
\end{itemize}
\end{definition}
\begin{definition}[Interpretation]\label{interpretation}
An \textbf{interpretation} $I$ consists of
\begin{itemize}
\item a set $D$, called the domain of interpretation;
\item an element $I(c_i)\in D$ for each constant $c_i$;
\item a function $I(f^i_j):D^i\longrightarrow D$ for each function symbol $f^i_j$;
\item an L-fuzzy relation $I(p^i_j):D^i\longrightarrow L$, where $L$ is a frame, for each predicate symbol $p^i_j$ i.e. it is an L-fuzzy subset of $D^i$.
\end{itemize}
\end{definition}
\begin{definition}[Graded Satisfiability]\label{sat}
Let $s$ be a sequence over $D$. Let $s=(s_1,s_2,\dots)$ be a sequence over $D$ where $s_1,s_2,\dots$ are all elements of $D$. Let $d$ be an element of $D$. Then $s(d/x_i)$ is the result of replacing $i$'th coordinate of $s$ by $d$ i.e., $s(d/x_i)=(s_1,s_2,\dots,s_{i-1},d,s_{i+1},\dots)$. Let $t$ be a term. Then $s$ assigns an element $s(t)$ of $D$ as follows:
\begin{itemize}
\item if $t$ is the constant symbol $c_i$ then $s(c_i)=I(c_i)$;
\item if $t$ is the variable $x_i$ then $s(x_i)=s_i$;
\item if $t$ is the function symbol $f^i_jt_1t_2\dots t_i$ then\\ $s(f^i_jt_1t_2\dots t_i)=I(f^i_j)(s(t_1),s(t_2),\dots,s(t_i))$.
\end{itemize}
Now we define grade of satisfiability of $\phi$ by $s$ written as $gr(s\ \emph{sat}\ \phi)$, where $\phi$ is a geometric formula, as follows:
\begin{itemize}
\item $gr(s\ \emph{sat}\ p^i_jt_1t_2\dots t_i)=I(p^i_j)(s(t_1),s(t_2),\dots,s(t_i))$;
\item $gr(s\ \emph{sat}\ \top)=1_L$;
\item $gr(s\ \emph{sat}\ \bot)=0_L$;
\item $gr(s\ \emph{sat}\ t_i=t_j)$ $=\begin{cases}
        1_L & \emph{if $s(t_i)=s(t_j)$} \\
        0_L & \emph{otherwise};
    \end{cases}$
\item $gr(s\ \emph{sat}\ \phi_1\wedge\phi_2)=gr(s\ \emph{sat}\ \phi_1)\wedge gr(s\ \emph{sat}\ \phi_2)$;
\item $gr(s\ \emph{sat}\ \phi_1\vee\phi_2)=gr(s\ \emph{sat}\ \phi_1)\vee gr(s\ \emph{sat}\ \phi_2)$;
\item $gr(s\ \emph{sat}\ \bigvee\{\phi_i\}_{i\in I})=sup\{gr(s\ \emph{sat}\ \phi_i)\mid i\in I\}$;
\item $gr(s\ \emph{sat}\ \exists x_i\phi)=sup\{gr(s(d/ x_i)\ \emph{sat}\ \phi)\mid d\in D\}$.
\end{itemize}
\end{definition}
 Throughout this article $\wedge$ and $\vee$ in $L$ will stand for the meet and join of the frame $L$ respectively.
The expression $\phi\vdash\psi$, where $\phi$, $\psi$ are wffs, is called a sequent. We now define satisfiability of a sequent.
\begin{definition}\label{valid}
1. $s$ \emph{sat} $\phi\vdash\psi$ iff $gr(s\ \emph{sat}\ \phi)\leq gr(s\ \emph{sat}\ \psi)$.\\
2. $\phi\vdash\psi$ is valid in $I$ iff $s$ \emph{sat} $\phi\vdash\psi$ for all $s$ in the domain of $I$.\\
3. $\phi\vdash\psi$ is universally valid iff it is valid in all interpretations.
\end{definition}
\begin{theorem}\label{LDI}
Let $I$ be an interpretation and $t$ be a term. If the sequences $s$ and $s'$ are such that they agree on the variables occurring in the term $t$ then $s(t)=s'(t)$.
\end{theorem}
\begin{proof}
By induction on $t$.
\end{proof}
\begin{theorem}\label{LDII}
Let $I$ be an interpretation and $\phi$ be a geometric formula. If the sequences $s$ and $s'$ are such that they agree on the free variables occurring in $\phi$ then $gr(s\ \emph{sat}\ \phi)=gr(s'\ \emph{sat}\ \phi)$.
\end{theorem}
\begin{proof}
By induction on $\phi$.
\end{proof}
\begin{theorem}[Substitution Theorem]\label{substitution}
Let $D$ be the domain of interpretation $I$:
\begin{enumerate}
\item Let $t$ and $t'$ be terms. For every sequence $s$ over $D$,\\ $s(t[t'/x_k])=s(s(t')/x_k)(t)$.
\item Let $\phi$ be a geometric formula and $t$ be a term. For every sequence $s$ over $D$, $gr(s\ \emph{sat}\ \phi[t/x_k])=gr(s(s(t)/x_k)\ \emph{sat}\ \phi)$.
\end{enumerate}
\end{theorem}
\begin{proof}
By induction on $t$ and $\phi$ respectively.
\end{proof}
\subsection{Rules of Inference}
In this subsection the rules of inference for generalised geometric logic are given. A rule of inference for generalised geometric logic is of the form \\\AxiomC{$\mathscr{S}_1,\mathscr{S}_2,\dots ,\mathscr{S}_i$}
\UnaryInfC{$\mathscr{S}$}
\DisplayProof, where each of the $\mathscr{S}_1,\mathscr{S}_2,\dots,\mathscr{S}_i$ and $\mathscr{S}$ is a sequent. The sequents $\mathscr{S}_1,\mathscr{S}_2,\dots,\mathscr{S}_i$ are known as premises and the sequent $\mathscr{S}$ is called the conclusion. It should be noted that for a rule of inference the set of premises can be empty also.\\
The rules of inference for generalised geometric logic are as follows.
\begin{enumerate}
\item 
$\phi\vdash\phi$,
\item
\AxiomC{$\phi\vdash \psi$}
\AxiomC{$\psi\vdash\chi$}
\BinaryInfC{$\phi\vdash\chi$}
\DisplayProof ,
\item
(i) $\phi\vdash\top$,\ \ \ \  
(ii) $\phi\wedge\psi\vdash\phi$,\ \ \ \ 
(iii) $\phi\wedge\psi\vdash\psi$,\ \ \ \ 
(iv)\AxiomC{$\phi\vdash\psi$}
\AxiomC{$\phi\vdash\chi$}
\BinaryInfC{$\phi\vdash\psi\wedge\chi$}
\DisplayProof ,
\item 
(i) $\phi\vdash\bigvee S$ ($\phi\in S$),\hspace{24pt}
(ii)\AxiomC{$\phi\vdash\psi$}
\AxiomC{all $\phi\in S$}
\BinaryInfC{$\bigvee S\vdash \psi$}
\DisplayProof ,
\item
$\phi\wedge\bigvee S\vdash\bigvee\{\phi\wedge\psi\mid\psi\in S\}$,
\item
$\top\vdash (x=x)$,
\item
$((x_1,\dots,x_n)=(y_1,\dots,y_n))\wedge\phi\vdash\phi[(y_1,\dots,y_n)\mid(x_1,\dots,x_n)]$,
\item
(i)\AxiomC{$\phi\vdash\psi[x\mid y]$}
\UnaryInfC{$\phi\vdash\exists y\psi$}
\DisplayProof ,
\hspace{24pt}
(ii)\AxiomC{$\exists y\phi\vdash\psi$}
\UnaryInfC{$\phi[x\mid y]\vdash\psi$}
\DisplayProof ,
\item
$\phi\wedge (\exists y)\psi\vdash(\exists y)(\phi\wedge\psi)$.
\end{enumerate}
\subsection{Soundness}
The soundness of a rule means that if all the premises are valid in an interpretation $I$ then the conclusion must also valid in the same interpretation $I$. Satisfaction relation being many-valued, the validity of a sequent has a meaning different from that in the classical geometric logic. In this subsection we will show the soundness of the above rules of inference.
\begin{theorem}
The rules of inference for generalised geometric logic are universally valid.
\end{theorem}
\begin{proof}
\begin{enumerate}
\item $gr(s\ \text{sat}\ \phi)=gr(s\ \text{sat}\ \phi)$, for any $s$.
Hence $\phi\vdash\phi$ is valid.
\item Given $\phi\vdash \psi$ and $\psi\vdash\chi$ are valid.
So $gr(s\ \text{sat}\ \phi)\leq gr(s\ \text{sat}\ \psi)$ and $gr(s\ \text{sat}\ \psi)\leq gr(s\ \text{sat}\ \chi)$ for any $s$. Therefore $gr(s\ \text{sat}\ \phi)\leq gr(s\ \text{sat}\ \chi)$ for any $s$. Hence $\phi\vdash\chi$ is valid when $\phi\vdash\psi$ and $\psi\vdash\chi$  are valid.
\item (i) $gr(s\ \text{sat}\ \phi)\leq 1_L=gr(s\ \text{sat}\ \top)$ for any $s$.
Hence $\phi\vdash \top$ is valid.\\
(ii) $gr(s\ \text{sat}\ \phi\wedge\psi)=gr(s\ \text{sat}\ \phi)\wedge gr(s\ \text{sat}\ \psi)\leq gr(s\ \text{sat}\ \phi)$ for any $s$.
Hence $\phi\wedge\psi\vdash\phi$ is valid.\\
(iii) $gr(s\ \text{sat}\ \phi\wedge\psi)=gr(s\ \text{sat}\ \phi)\wedge gr(s\ \text{sat}\ \psi)\leq gr(s\ \text{sat}\ \psi)$ for any $s$.
Hence $\phi\wedge\psi\vdash\psi$ is valid.\\
(iv) Given $\phi\vdash\psi$ and $\phi\vdash\chi$ are valid. So $gr(s\ \text{sat}\ \phi)\leq gr(s\ \text{sat}\ \psi)$ and $gr(s\ \text{sat}\ \phi)\leq gr(s\ \text{sat}\ \chi)$ for any $s$. So $gr(s\ \text{sat}\ \phi)\leq gr(s\ \text{sat}\ \psi)\wedge gr(s\ \text{sat}\ \chi)=gr(s\ \text{sat}\ \psi\wedge\chi)$ for any $s$. Hence $\phi\vdash \psi\wedge \chi$ is valid when $\phi\vdash\psi$ and $\phi\vdash\chi$ are valid.
\item (i) $gr(s\ \text{sat}\ \phi)\leq gr(s\ \text{sat}\ \bigvee S(\phi\in S))$ for any $s$. Hence $\phi\vdash \bigvee S (\phi\in S)$ is valid.
(ii) Given $\phi\vdash \psi$ is valid for all $\phi\in S$.
So $gr(s\ \text{sat}\ \phi)\leq gr(s\ \text{sat}\ \psi)$ for all $\phi\in S$ and any $s$.
So, $sup_{\phi\in S}\{gr(s\ \text{sat}\ \phi)\}\leq gr(s\ \text{sat}\ \psi)$ for any $s$. Hence $gr(s\ \text{sat}\ \bigvee S)\leq gr(s\ \text{sat}\ \psi)$ for any $s$. So, $\bigvee S\vdash\psi$ is valid when $\phi\vdash \psi$ is valid for all $\phi\in S$.
\item We have, 
$gr(s\ \text{sat}\ \phi\wedge\bigvee S) = gr(s\ \text{sat}\ \phi)\wedge gr(s\ \text{sat}\ \bigvee S) = gr(s\ \text{sat}\ \phi)\wedge sup_{\psi\in S}\{gr(s\ \text{sat}\ \psi)\} = sup_{\psi\in S}\{gr(s\ \text{sat}\ \phi)\wedge gr(s\ \text{sat}\ \psi)\} = sup\{gr(s\ \text{sat}\ \phi\wedge\psi)\mid\psi\in S\},\ \ \text{for any}\ s$.
Hence $\phi\wedge\bigvee S\vdash sup\{\phi\wedge\psi\mid\psi\in S\}$ is valid.
\item $gr(s\ \text{sat}\ \top)=1_L=gr(s\ \text{sat}\ x=x)$, for any $s$.
Hence $\top\vdash x=x$ is valid.
\item $gr(s\ \text{sat}\ ((x_1,\dots,x_n)=(y_1,\dots,y_n))\wedge\phi)$\\
$=gr(s\ \text{sat}\ ((x_1,\dots,x_n)=(y_1,\dots,y_n)))\wedge gr(s\ \text{sat}\ \phi)$.\\
Now $gr(s\ \text{sat}\ \phi[(y_1,\dots,y_n)/(x_1,\dots,x_n)])$\\
$=gr(s(s((y_1,\dots,y_n))/(x_1,\dots,x_n))\ \text{sat}\ \phi)$.\\
When $s((y_1,\dots,y_n))=s((x_1,\dots,x_n))$ \\
then $gr(s(s((y_1,\dots,y_n))/(x_1,\dots,x_n))\ \text{sat}\ \phi)=gr(s\ \text{sat}\ \phi)$.
\\ Hence, $gr(s\ \text{sat}\ ((x_1,\dots,x_n)=(y_1,\dots,y_n))\wedge\phi)$\\
$\leq gr(s\ \text{sat}\ \phi[(y_1,\dots,y_n)/(x_1,\dots,x_n)])$, for any $s$.\\
So, $((x_1,\dots,x_n)=(y_1,\dots,y_n))\wedge\phi\vdash\phi[(y_1,\dots,y_n)/(x_1,\dots,x_n)]$ is valid.
\item (i) $\phi\vdash \psi[x\mid y]$ is valid so, $gr(s\ \text{sat}\ \phi)\leq gr(s\ \text{sat}\ \psi[x\mid y])$, for any $s$. Using Theorem \ref{substitution}(2) $gr(s\ \text{sat}\ \phi)\leq gr(s(s(x)/y)\ \text{sat}\ \psi)$, for any $s$, which implies that $gr(s\ \text{sat}\ \phi)\leq sup\{gr(s(d/y)\ \text{sat}\ \psi)\mid d\in D\}$, for any $s$. So, $gr(s\ \text{sat}\ \phi)\leq gr(s\ \text{sat}\ \exists y\psi)$ and hence $\phi\vdash\exists y\psi$ is valid.\\
(ii) $\exists y\phi\vdash \psi$ is valid if and only if $gr(s\ \text{sat}\ \exists y\phi)\leq gr(s\ \text{sat}\ \psi)$, for any $s$. Hence $sup\{gr(s(d/y)\ \text{sat}\ \phi)\mid d\in D\}\leq gr(s\ \text{sat}\ \psi)$, for any $s$. So, $gr(s(s(x)/y)\ \text{sat}\ \phi)\leq gr(s\ \text{sat}\ \psi)$, for any $s$, using Theorem \ref{substitution}(2). Therefore $gr(s\ \text{sat}\ \phi[x/y])\leq gr(s\ \text{sat}\ \psi)$, for any $s$ and hence $\phi[x/y]\vdash\psi$ is valid provided $\exists y\phi\vdash\psi$ is valid.
\item $gr(s\ \text{sat}\ \phi\wedge (\exists y)\psi)$
$=gr(s\ \text{sat}\ \phi)\wedge gr(s\ \text{sat}\ \exists y\psi)$
$=gr(s\ \text{sat}\ \phi)\wedge sup_{d\in D}\{gr(s(d/y)\ \text{sat}\ \psi)\}$
$=sup_{d\in D}\{gr(s\ \text{sat}\ \phi)\wedge gr(s(d/y)\ \text{sat}\ \psi)\}$
$\leq sup_{d\in D}\{gr(s(d/y)\ \text{sat}\ \phi)\wedge gr(s(d/y)\ \text{sat}\ \psi)\}$ 
$=sup_{d\in D}\{gr(s\ \text{sat}\ \phi\wedge\psi)\}$
$=gr(s\ \text{sat}\ (\exists y)\phi\wedge\psi)$, for any $s$.
Hence $\phi\wedge(\exists y)\psi\vdash(\exists y)(\phi\wedge\psi)$ is valid.
\end{enumerate}
\end{proof}

\section{$L$-Topological System via Generalised Geometric Logic}\label{ltsy}
Let us consider the triplet $(X,\models,A)$ where $X$ is the non empty set of assignments $s$, $A$ is the set of geometric formulae and $\models$ defined as $gr(s\models \phi)=gr(s\ \text{sat}\ \phi)$.
\begin{theorem}\label{fts}
(i) $gr(s\models \phi\wedge\psi)=gr(s\models \phi)\wedge gr(s\models\psi)$.
\\(ii) $gr(s\models\bigvee\{\phi_i\}_{i\in I})=sup_{i\in I}\{gr(x\models \phi_i)\}$.
\end{theorem}
\begin{proof}
(i) $gr(s\models \phi\wedge\psi)=gr(s\ \text{sat}\ \phi\wedge\psi)=gr(s\ \text{sat}\ \phi)\wedge gr(s\ \text{sat}\ \psi)=gr(s\models \phi)\wedge gr(s\models \psi)$.
(ii) $gr(s\models\bigvee\{\phi_i\}_{i\in I})=gr(s\ \text{sat}\ \bigvee\{\phi_i\}_{i\in I})=sup_{i\in I}\{gr(s\ \text{sat}\ \phi_i)\}=sup_{i\in I}\{gr(s\models\phi_i)\}$.
\end{proof} 
\begin{definition}\label{equiv}
$\phi\approx\psi$ iff $gr(s\models \phi)=gr(s\models \psi)$ for any $s\in X$ and $\phi,\psi\in A$.
\end{definition}
The above defined ``$\approx$" is an equivalence relation. Thus we get $A/_{\approx}$.
\begin{theorem}\label{ftopsys}
$(X,\models ',A/_{\approx})$ is an $L$-topological system, where $\models '$ is defined by $gr(s\models '[\phi])=gr(s\models \phi)$.
\end{theorem}
\begin{proof}
$X$ is a non empty set of assignments $s$.
Let us first prove that $A/_{\approx}$ is a frame in the following way. Here we define $[\phi]\leq [\psi]$ as follows: $[\phi]\leq [\psi]\  \text{iff}\ \ gr(s\models\phi)\leq gr(s\models \psi)\ \ \ \text{for any}\ s$ i.e., $\phi\vdash\psi$ is valid.
Now in generalised geometric logic $\phi\vdash\phi$ is valid and if $\phi\vdash\psi$ and $\psi\vdash\chi$ are valid then $\phi\vdash\chi$ is valid. Thus $\leq$ is reflexive and transitive. If $[\phi]\leq [\psi]$ and $[\psi]\leq [\phi]$ then $gr(s\models \phi)\leq gr(s\models \psi)$ and $gr(s\models \psi)\leq gr(s\models \phi)$ for any $s$. Therefore $gr(s\models \phi)=gr(s\models \psi)$ for any $s$. So $\phi\approx\psi$. Consequently $[\phi]=[\psi]$. Hence $A/_{\approx}$ is a poset. Now if $\phi,\psi\in A$ then $\phi\wedge\psi\in A$ (by Theorem \ref{fts}). So $[\phi],[\psi]\in A/_{\approx}$ and $[\phi\wedge\psi]\in A/_{\approx}$ i.e., $[\phi]\wedge [\psi]\in A/_{\approx}$. Similarly arbitrary join exists in $A/_{\approx}$. $[\phi]\wedge\bigvee\{[\psi_i]\}_{i\in I}=[\phi]\wedge [\bigvee\{\psi_i\}_{i\in I}]=[\phi\wedge\bigvee\{\psi_i\}_{i\in I}]$ Now we have $\phi\wedge\bigvee\{\psi_i\}_{i\in I}\vdash sup_{i\in I}\{\phi\wedge\psi_i\}$ is valid. Hence $gr(s\ \text{sat}\ \phi\wedge\bigvee\{\psi_i\}_{i\in I})\leq gr(s\ \text{sat}\ \bigvee\{\phi\wedge\psi_i\}_{i\in I})$ for any $s$. $sup_{i\in I}\{\phi\wedge\psi_i\}\vdash \phi\wedge\bigvee\{\psi_i\}_{i\in I}$ is derivable, so $gr(s\ \text{sat}\ \bigvee\{\phi\wedge\psi_i\}_{i\in I})\leq gr(s\ \text{sat}\ \phi\wedge\bigvee\{\psi_i\}_{i\in I})$ for any $s$. Therefore $gr(s\ \text{sat}\ \phi\wedge\bigvee\{\psi_i\}_{i\in I})=gr(s\ \text{sat}\ \bigvee\{\phi\wedge\psi_i\}_{i\in I})$ for any $s$. So, $[\phi\wedge\bigvee\{\psi_i\}_{i\in I}]=[\bigvee\{\phi\wedge\psi_i\}_{i\in I}]$. Hence $[\phi]\wedge\bigvee\{[\psi_i]\}_{i\in I}=[\bigvee\{\phi\wedge\psi_i\}_{i\in I}]=\bigvee\{[\phi\wedge\psi_i]\}_{i\in I}=\bigvee\{([\phi]\wedge [\psi_i])\}_{i\in I}$. Hence $A/_{\approx}$ is a frame.\\
Now it is left to show that (a) $gr(s\models ' [\phi]\wedge [\psi])=gr(s\models '[\phi])\wedge gr(s\models ' [\psi])$ and (b) $gr(s\models '\bigvee\{[\phi_i]\}_{i\in I})=sup_{i\in I}\{gr(s\models '[\phi_i])\}$.
\\Proof of the above follow easily using Theorem \ref{fts}. Hence $(X,\models ',A/_{\approx})$ is an $L$-topological system.
\end{proof}
\begin{proposition}\label{spec}
In the $L$-topological system $(X,\models ',A/_{\approx})$, defined above, for all $s\in X$, $(gr(s\models '[\phi])=gr(s\models '[\psi]))\ \text{implies}\ ([\phi]=[\psi])$.
\end{proposition}
\begin{proof}
As $gr(s\models '[\phi])=gr(s\models '[\psi])$, for any $s$, we have $gr(s\models \phi)=gr(s\models \psi)$, for any $s$. Hence $\phi\approx\psi$ and consequently $[\phi]=[\psi]$.
\end{proof}
\section{$L$-Topology via Generalised Geometric Logic}\label{lts}
We first construct the $L$-topological system $(X,\models ',A/_{\approx})$ from generalised geometric logic. Then $(X,ext(A/_{\approx}))$ is constructed as follows:
\\ $ext(A/_{\approx})=\{ext([\phi])\}_{[\phi]\in A/_{\approx}}$ where
$ext([\phi]):X\longrightarrow L$ is such that, for each $[\phi]\in A/_{\approx}$, $ext([\phi])(s)=gr(s\models '[\phi])=gr(s\models \phi)$.\\
It can be shown that $ext(A/_{\approx})$ forms an $L$-topology on $X$ as follows.\\
Let $ext([\phi]),ext([\psi])\in ext(A/_{\approx})$. 
Then $(ext([\phi ])\cap ext([\psi ]))(s)= \\(ext([\phi ]))(s)\wedge (ext([\psi ]))(s)= gr(s\models ' [\phi ])\wedge gr(s\models ' [\psi ]) = gr(s\models \phi )\wedge gr(s\models \psi ) = gr(s\models \phi\wedge \psi ) = gr(s\models '[\phi \wedge \psi ]) = (ext([\phi \wedge \psi ]))(s)$. Hence $ext([\phi])\cap ext([\psi])=ext([\phi\wedge \psi])\in ext(A/_{\approx})$. Similarly it can be shown that $ext(A/_{\approx})$ is closed under arbitrary union. Hence $(X,ext(A/_{\approx}))$ is an $L$-topological space obtained via generalised geometric logic.\\

Proposition \ref{spec} indicates that $(X,\models',A/_{\approx})$ is a spatial $L$-topological system and hence from Theorem \ref{spatialtop} we arrive at the conclusion that $(X,\models',A/_{\approx})$, $(A,ext(A/_{\approx}))$ are equivalent to each other. That is, $(X,\models',A/_{\approx})$ and $(X,\in,ext(A/_{\approx}))$ represent the same $L$-topological system.

Let $X$ be an $L$-topological space, $\tau$ is its $L$-topology. Then the corresponding generalised geometric theory can be defined as follows:
\begin{itemize}
    \item for each $L$-open set $\tilde{T}$, a proposition $P_{\tilde{T}}$
    \item if $\tilde{T_1}\subseteq \tilde{T_2}$, then an axiom $$P_{\tilde{T_1}}\vdash P_{\tilde{T_2}}$$
    \item if $S$ is a family of $L$-open sets, then an axiom 
    $$P_{\bigcup S}\vdash\bigvee_{\tilde{T}\in S}P_{\tilde{T}}$$
    \item if $S$ is finite collection of $L$-open sets, then an axiom
    $$\bigwedge_{\tilde{T}\in S}P_{\tilde{T}}\vdash P_{\bigcap S}$$
\end{itemize}
All other axioms for the (propositional) generalised geometric logic  will follow from the above clauses.

If $x\in X$, then $x$ gives a model of the theory in which the truth value of the  interpretation of $P_{\tilde{T}}$ will be $\tilde{T}(x)$.

Hence one may study $L$-topology via generalised geometric logic.
\section{Concluding Remarks}\label{con}
In this paper the notion of generalised geometric logic is introduced and studied in details. Using the connection between $L$-topological system and $L$-topological space, the strong connection between the proposed logic and $L$-topological space is established. The interpretation of the predicate symbols for the generalised geometric logic are $L$ (frame)-valued relations and so the proposed logic is more expressible. That is, the proposed logic has the capacity to interpret the situation where the truth values are incomparable. Generalising the proposed logic considering graded consequence relation is in future goal which will appear in our next paper. 
\section*{References}

\bibliography{mybibfile}

\end{document}